\documentclass[12pt,a4paper]{article}
\usepackage{amsmath,amssymb,amsfonts}
\usepackage{amsthm}
\usepackage{algorithm,algorithmic}
\usepackage[figuresright]{rotating}
\usepackage{indentfirst}
\usepackage{graphicx}
\usepackage{epstopdf}
\usepackage[colorlinks,linkcolor=blue,anchorcolor=green,citecolor=red]{hyperref}
\usepackage{bookmark}
\usepackage{multirow,booktabs,array}

\newtheorem{theorem}{Theorem}[section]
\newtheorem{lemma}{Lemma}[section]

\title{A modified adaptive cubic regularization method for large-scale unconstrained optimization problem\thanks{This work was supported by the National Science Foundation of China under Grant No. 11571004.}}
\author{Yutao Zheng$^{a,b}$\thanks{Corresponding author. E-mail address: zhengyutao@htu.edu.cn.} and Bing Zheng$^a$\\ {\small $^a$ School of Mathematics and Statistics, Lanzhou University}\\ {\small Lanzhou 730000, P.R. China}\\
{\small $^b$ College of Mathematics and Information Science, Henan Normal University}\\{\small Xinxiang 453007, China}}
\date{}
\begin{document}
\maketitle
\begin{quote}
{\bf Abstract}
 {\small In this paper, we modify the adaptive cubic regularization method for large-scale  unconstrained optimization problem by using a real positive definite scalar matrix to approximate the exact Hessian. Combining with the nonmonotone technique, we also give a variant of the modified algorithm. Under some reasonable conditions, we analyze the global convergence of the proposed methods. Numerical experiments are performed and the obtained results show satisfactory performance when compared to the standard trust region method, adaptive regularization algorithm with cubics and the simple model trust-region method.}


\end{quote}
\section{Introduction}
Consider the following unconstrained optimization problem
\begin{equation}\label{pro}
\min f(x), x\in \mathbb{R}^n
\end{equation}
where $f:\mathbb{R}^n\rightarrow \mathbb{R}$ is a continuously differentiable function and its gradient $g(x)$ is available. In addition to the trust-region (TR) \cite{Conn1987} and line-search \cite{Nocedal} methods, Cartis et al. \cite{Cartis2011} suggested a third alternative for solving \eqref{pro} --- the use of a cubic overestimator of the objective function as a regularization technique.

Let $x_k$ be the current iterate point, $g_k=g(x_k)$ and $B_k$ denotes the exact Hessian of $f$ at $x_k$ or its symmetric approximation. According to the adaptive regularization algorithm using cubics (ARC) proposed in \cite{Cartis2011}, at each iteration $k$, the cubic model is
\begin{equation}\label{subprob}
m_k(s)=f(x_k)+s^Tg_k+\frac12s^TB_ks+\frac13\sigma_k\|s\|^3,
\end{equation}
where $\sigma_k>0$ is an adaptive parameter which can be viewed as the reciprocal of the trust-region radius. The trial step $s_k$ is computed as an approximate global minimizer of $m_k(s)$. The next iterate $x_k+s_k$ is accepted if the value of the metric function
\begin{equation*}
\rho_k=\frac{f(x_k)-f(x_k+s_k)}{f(x_k)-m_k(s_k)}
\end{equation*}
is greater than some positive constant of $\eta_1$. The adaptive parameter $\sigma_k$ is updated by a similar mechanism in the trust region method.

Compared to the standard trust region approach, the ARC method has a better numerical performance for small- and medium-scale test problems, see \cite{Cartis2011} for more details. Another good point for ARC method is its worst-case complexity when the exact second order information is provided, which means that we can find an $\epsilon$-approximate first-order critical point for an unconstrained problem with Lipschitz continuous Hessians in at most $O(\epsilon^{-2/3})$ evaluations of the objective function (and its derivatives).

Since then, more and more researchers take their attention to this new topic. Bianconcini et al. \cite{Bianconcini2015} supplied a viable alternative to the implementation of ARC method using the GLRT routine of GALAHAD library \cite{Gould2003a}. Bergou et al. \cite{Bergou2017} considered the energy norm of $\|s\|_{M}=\sqrt{s^TMs}$ instead of the general Euclidean norm $\|s\|$, where $M$ is a symmetric positive definite matrix. Bianconcini and Sciandrone \cite{Bianconcini2016a} introduced the line search and nonmonotone techniques to the ARC algorithm. Some worst-case evaluation complexity results can be found in \cite{Birgin2017a,Cartis2017a,Cartis2011a}. We refer the readers to \cite{Benson2018,Dussault2017,Gould2011,Martinez2017b,Martinez2017} and references therein for more related work.

However, from the view of numerical simulation, the TR or ARC method will be not suitable for large-scale problems, it may be very expensive to evaluate the exact Hessians. Moreover, choosing the true or a quasi-Newton approximate Hessian as $B_k$ will make finding the minimizer of the cubic model \eqref{subprob} more difficult and more complex, see more details in \cite{Cartis2011}. Such of these drawbacks can restrict the applications of the TR and ARC methods in practice.

Recently, Zhou et al. \cite{Zhou2015} introduced a nonmonotone adaptive trust region method with line search based on diagonal updating the Hessians, then Zhou et al. \cite{Zhou2016} embedded the Barzilai-Borwein step length \cite{Barzilai1988} in the framework of simple model trust-region method. The resulted algorithm simplifies the computation of solving the trust-region subproblems and shows better numerical performance when compared with the global Barzilai-Borwein method \cite{Raydan1997}.

In this paper, we will take $B_k$ as a real positive definite scalar matrix $\gamma_kI$, where $\gamma_k>0$ is the Barzilai-Borwein step length or some of its variants, which will inherit some certain quasi-Newton property. The minimizer of the resulted cubic model can be easily determined, and at the same time, the convergence of the algorithm is also   maintained.

The rest of this paper is organized as follows. In Sect. \ref{sec2}, we propose a modified adaptive regularization algorithm by introducing the Barzilai-Borwein parameter. Under some certain conditions, the global and strong convergence of the modified method is studied in Sect. \ref{sec3}. With nonmonotone technique, we also present a variant of the MARC algorithm and analyze the corresponding convergence in Sect. \ref{sec4}. Numerical experiments are performed in Sect. \ref{sec5}. Finally, in Sect. \ref{sec6}, we give some conclusions to end this paper.

\section{The MARC Method}\label{sec2}
We first review the Barzilai-Borwein gradient method \cite{Barzilai1988} which is given by
\begin{equation*}
x_{k+1}=x_k-D_kg_k,
\end{equation*}
where $D_k=\frac1{\gamma_k}I$. Ask $D_k$ to have the certain quasi-Newton property, that is,
\begin{equation}\label{gamma}
\gamma_k=\arg\min\limits_{\gamma}\|s_{k-1}-\frac1{\gamma}y_{k-1}\|=\frac{s_{k-1}^Ty_{k-1}}{s_{k-1}^Ts_{k-1}},
\end{equation}
where $s_{k-1}=x_k-x_{k-1}$ and $y_k=g_k-g_{k-1}$.

Due to its very satisfactory performance compared to the steepest descent method and better convergence property, Barzilai-Borwein gradient method now has been developed into a competitive method for large scale problems and got broad attention of numerous experts. Moreover, the Barzilai-Borwein step length is also widely applied in other optimization algorithms, such as the trust region method \cite{Zhou2016} and the conjugate gradient method \cite{Dai2013}, and was used to solve constrained optimization problem \cite{Dai2005a,Huang2015a}, multiobjective optimization problem \cite{Morovati2016} and so on.

By using the real positive definite scalar matrix $\gamma_kI$ to approximate the Hessian, where $\gamma_k$ is defined by \eqref{gamma} as before, the subproblem \eqref{subprob} in the ARC algorithm can be written as
\begin{equation}\label{subpro}
\min m_k(s)=f(x_k)+g_k^Ts+\frac12\gamma_ks^Ts+\frac13\sigma_k\|s\|^3, s\in \mathbb{R}^n,
\end{equation}
which can be easily solved.

Since $\nabla m_k(s)=g_k+\gamma_ks+\sigma_k\|s\|s$, the solution $s_k$ of problem \eqref{subpro} is parallel to $-g_k$. Let $s_k=-\alpha_k g_k$, \eqref{subpro} is equivalent to
\begin{equation*}
\min\limits_{\alpha>0}\phi(\alpha)=f(x_k)-\|g_k\|^2\alpha+\frac12\gamma_k\|g_k\|^2\alpha^2+\frac13\sigma_k\|g_k\|^3\alpha^3.
\end{equation*}
Asking $\phi^\prime(\alpha)=\|g_k\|^2(\sigma_k\|g_k\|\cdot\alpha^2+\gamma_k\cdot\alpha-1)=0$ yields a positive solution of
\begin{equation*}
\alpha_k=\frac{-\gamma_k+\sqrt{\gamma_k^2+4\sigma_k\|g_k\|}}{2\sigma_k\|g_k\|}=\frac2{\gamma_k+\sqrt{\gamma_k^2+4\sigma_k\|g_k\|}}.
\end{equation*}
Thus the exact solution of \eqref{subpro} is
\begin{equation}\label{sk}
s_k=-\frac2{\gamma_k+\sqrt{\gamma_k^2+4\sigma_k\|g_k\|}}g_k,
\end{equation}
which implies that
\begin{equation}\label{ssk}
\|s_k\|\leq \sqrt{\frac{\|g_k\|}{\sigma_k}}.
\end{equation}

Since the Barzilai-Borwein steplength $\gamma_k$ would be very large or negative, the truncation technique will be used, we restrict $\gamma_k$ to the interval $[\gamma_{\min},\gamma_{\max}]$. Now, recall the framework of the adaptive cubic regularization algorithm and the trust region method, we state our algorithm as below:
\begin{algorithm}[H]
    \caption{Modified Adaptive Regularization algorithm using Cubics (MARC)}
    \label{marc}
    \begin{algorithmic}[1]
        \REQUIRE $x_0$, $\eta_2\geq\eta_1>0$, $c_1\geq1\geq c_2>0$, $\gamma_{\max}>\gamma_0>\gamma_{\min}>0$ and $\sigma_0>0$.\\ For $k=0,1,\cdots$ until convergence,
        \STATE Compute $s_k$ by \eqref{sk}.
        \STATE Compute $f(x_k+s_k)$ and
        \begin{equation}\label{rhok}
        \rho_k=\frac{f(x_k)-f(x_k+s_k)}{f(x_k)-m_k(s_k)};
        \end{equation}
        \STATE Set
        \begin{align*}
        x_{k+1}=\left\{\begin{array}{ll}x_k+s_k &\text{if}~\rho_k\geq \eta_1,\\x_k&\text{otherwise}.\end{array}\right.
        \end{align*}
        \STATE Set
        \begin{equation}\label{update}
        \sigma_{k+1}=\left\{
        \begin{array}{lll}
        c_2\sigma_k, & \text{if}~\rho_k>\eta_2, & \text{[very successful iteration]},\\
        \sigma_k, & \text{if}~\eta_1\leq\rho_k\leq\eta_2, & \text{[successful iteration]},\\
        {c_1\sigma_k}, & \text{otherwise}, & \text{[unsuccessful iteration]}.
        \end{array}\right.
        \end{equation}
        \STATE Compute $\gamma_{k+1}$ by \eqref{gamma}, \eqref{bb1} or \eqref{bb2}. Set $\gamma_{k+1}=\max\{\gamma_{\min},\min\{\gamma_{k+1},\gamma_{\max}\}\}$.
    \end{algorithmic}
\end{algorithm}


Note that the scalar $\gamma_k$ plays an important role in Algorithm \ref{marc}, hence, another two formulas developed by Yabe et al. \cite{Yabe2007} and Zheng et al. \cite{Zheng2017} to compute $\gamma_k$ are considered, they are defined as

\begin{equation}\label{bb1}
\gamma_{k+1}=\frac{s_k^Ty_k+\theta_k[2(f_k-f_{k+1})]+(g_k+g_{k+1})^Ts_k}{s_k^Ts_k},
\end{equation}
where $\theta_k\in[0,3]$, and
\begin{equation}\label{bb2}
\gamma_{k+1}=\frac{r_k^Tw_k}{r_k^Tr_k},
\end{equation}
where $r_k=s_k-\psi_ks_{k-1}$ and $w_k=y_k-\psi_ky_{k-1}$, $\psi_k\geq0$ is a scalar. The $\gamma_{k+1}$ in \eqref{bb2} reduces to \eqref{gamma} if $\psi_k=0$ for all $k$.

\section{Convergence Analysis}\label{sec3}
Our convergence analysis is similar to that of the ARC method in \cite{Cartis2011}. However, the real positive scalar matrix, a special approximation of Hessian can simplify the analyses of some conclusions in some sense. Throughout this and the next section, we denote the index set of all successful iterations of the MARC algorithm by
\begin{equation*}
\mathcal{S}\triangleq\{k\geq0:k~\text{is successful or very successful}\}.
\end{equation*}

We first give a lower bound on the decrease in $f$ predicted from the cubic model.

\begin{lemma}\label{lem1}
    Let $f\in C^1$. Suppose that the step $s_k$ is computed by \eqref{sk}. Then for all $k\geq0$, we have that
    \begin{equation}\label{pred}
    f(x_k)-m_k(s_k)\geq\frac{\|g_k\|}{12}\min\Bigg\{\frac{\|g_k\|}{\gamma_k},\frac12\sqrt{\frac{\|g_k\|}{\sigma_k}}\Bigg\}.
    \end{equation}
\end{lemma}
\begin{proof}
    By a direct computation, we have
    \begin{align*}
    f(x_k)-m_k(s_k)=&f(x_k)-m_k(-\alpha_kg_k)\nonumber\\
    =&\alpha_k\|g_k\|^2\Big\{1-\frac12\gamma_k\alpha_k-\frac13\sigma_k\|g_k\|\alpha_k^2\Big\}\nonumber\\
    \geq&\frac{\|g_k\|^2}{6(\gamma_k+2\sqrt{\sigma_k\|g_k\|})}\nonumber\\
    \geq&\frac{\|g_k\|}{12}\min\Bigg\{\frac{\|g_k\|}{\gamma_k},\frac12\sqrt{\frac{\|g_k\|}{\sigma_k}}\Bigg\},
    \end{align*}
    since $\frac1{\gamma_k+2\sqrt{\sigma\|g_k\|}}\leq\frac1{\sqrt{\gamma_k^2+4\sigma\|g_k\|}}\leq\alpha_k\leq\frac1{\gamma_k}$ and $\alpha_k\leq\frac1{\sqrt{\sigma_k\|g_k\|}}$.
\end{proof}
In the above proof, we remark that $\alpha_k$ is also less than $\frac2{\frac12\gamma_k+\sqrt{\frac14\gamma_k^2+\frac43\sigma_k\|g_k\|}}$, which is the positive root of $h(\alpha)=1-\frac12\gamma_k\alpha-\frac13\sigma_k\|g_k\|\alpha^2=0$. This fact gives that $h(\alpha_k)>0$. Furthermore, $h(\alpha_k)>\frac16$ since $\alpha_k\leq\frac1{\gamma_k}$ gives $\gamma_k\alpha_k\leq1$ and $\alpha_k\leq\frac1{\sqrt{\sigma_k\|g_k\|}}$ gives $\sigma_k\|g_k\|\alpha_k^2\leq1$.

An auxiliary lemma is given next.
\begin{lemma}\label{lem2}
    Let $f\in C^1$. Suppose that $\epsilon>0$ and $\mathcal{I}$ is an infinite index set such that
    \begin{equation}\label{ass1}
    \|g_k\|\geq\epsilon, k\in \mathcal{I},
    \end{equation}
    and
    \begin{equation}\label{ass11}
    \sqrt{\frac{\|g_k\|}{\sigma_k}}\rightarrow 0, k\rightarrow \infty, k\in \mathcal{I}.
    \end{equation}
    Furthermore, if for some $x_*\in\mathbb{R}^n$,
    \begin{equation}\label{ass2}
    x_k\rightarrow x_*, k\rightarrow \infty, k\in \mathcal{I},
    \end{equation}
    then each iteration $k\in \mathcal{I}$ that is sufficiently large is very successful, and
    \begin{equation*}
    \sigma_{k+1}\leq\sigma_k
    \end{equation*}
    holds for all $k\in \mathcal{I}$ sufficiently large.
\end{lemma}
\begin{proof}
    We first estimate the difference between the function and the model at $x_k+s_k$. A Taylor expansion of $f(x_k+s_k)$ around $x_k$ gives
    \begin{equation*}
    f(x_k+s_k)-m_k(s_k)=\big(g(\xi_k)-g_k\big)^Ts_k-\frac12\gamma_k\|s_k\|^2-\frac13\sigma_k\|s_k\|^3
    \end{equation*}
    for some $\xi_k$ on the line segment $(x_k,x_k+s_k)$, which, together with \eqref{ssk}, further gives
    \begin{equation}\label{diff}
    f(x_k+s_k)-m_k(s_k)\leq\Bigg\{\|g(\xi_k)-g_k\|+\frac{\gamma_k}2\sqrt{\frac{\|g_k\|}{\sigma_k}}\Bigg\}\cdot\sqrt{\frac{\|g_k\|}{\sigma_k}}.
    \end{equation}
    On the other hand, it follows from \eqref{pred} and the boundedness of $\gamma_k$ that
    \begin{equation*}
    f(x_k)-m_k(s_k)\geq\frac{\|g_k\|}{12}\min\Bigg\{\frac{\|g_k\|}{\gamma_{\max}},\frac12\sqrt{\frac{\|g_k\|}{\sigma_k}}\Bigg\}
    \end{equation*}
    holds for all $k$. By employing the limit in \eqref{ass11}, for all $k\in\mathcal{I}$ sufficiently large, we have
    \begin{equation*}
    f(x_k)-m_k(s_k)\geq\frac{\|g_k\|}{24}\sqrt{\frac{\|g_k\|}{\sigma_k}}.
    \end{equation*}
    Thus, by \eqref{diff}, we obtain that
    \begin{align}\label{rk}
    r_k:=&f(x_k+s_k)-f(x_k)-\eta_2[m_k(s_k)-f(x_k)]\nonumber\\
    =&f(x_k+s_k)-m_k(s_k)+(1-\eta_2)[m_k(s_k)-f(x_k)]\nonumber\\
    \leq&\sqrt{\frac{\|g_k\|}{\sigma_k}}\Bigg\{\|g(\xi_k)-g_k\|+\frac{\gamma_{\max}}2\sqrt{\frac{\|g_k\|}{\sigma_k}}-\frac{(1-\eta_2)\epsilon}{24}\Bigg\},
    \end{align}
    for all $k\in\mathcal{I}$ sufficiently large.
    
    Since $\xi_k\in(x_k,x_k+s_k)$, we have $\|\xi_k-x_*\|\leq\|x_k-x_*\|+\|s_k\|$. Also, \eqref{ssk} and \eqref{ass11} imply that $\|s_k\|\rightarrow 0$ as $k\rightarrow \infty$, $k\in \mathcal{I}$. Therefore $\xi_k\rightarrow x_*$ as $k\rightarrow \infty$, $k\in \mathcal{I}$ due to \eqref{ass2}. Since $g$ is continuous, we conclude that $\|g(\xi_k)-g_k\|\rightarrow 0$ as $k\rightarrow \infty$, $k\in \mathcal{I}$. Thus, the limits in \eqref{ass11} and \eqref{rk} imply that $r_k<0$, for all $k\in \mathcal{I}$ sufficiently large. Recalling that
    \begin{equation*}
    r_k<0 \Leftrightarrow \rho_k=\frac{f(x_k)-f(x_k+s_k)}{f(x_k)-m_k(s_k)}>\eta_2,
    \end{equation*}
    which means that the $k$-th iteration is very successful, and hence $\sigma_{k+1}\leq\sigma_k$.
\end{proof}

The following lemma shows that if there are only finitely many successful iterations, then all later iterates are first-order critical points.

\begin{lemma}\label{lem3}
    Let $f\in C^1$. Suppose that there are only finitely many successful iterations. Then $x_k=x_*$ for all sufficiently large $k$ and $g(x_*)=0$.
\end{lemma}

\begin{proof}
    See the proof of Lemma 2.4 in \cite{Cartis2011}.
\end{proof}

Now, we are ready to show the global convergence of the MARC algorithm. We conclude that if $f$ is bounded from below, either $g_k=0$ for some finite $k$, or there is a subsequence of $\{g_k\}$ converging to zero.

\begin{theorem}\label{theo1}
    Let $f\in C^1$, $\{x_k\}$ is the sequence generated by Algorithm \ref{marc}. If $\{f(x_k)\}$ is bounded below, then
    \begin{equation}\label{them1}
    \liminf_{k\rightarrow\infty}\|g_k\|=0.
    \end{equation}
\end{theorem}
\begin{proof}
    Lemma \ref{lem3} shows that \eqref{them1} is true when there are only finitely many successful iterations. Now we assume that infinitely many successful iterations occur, and recall the notation $\mathcal{S}$.
    
    We prove \eqref{them1} by contradiction. Assume that there is some $\epsilon>0$ such that
    \begin{equation}\label{gk}
    \|g_k\|\geq\epsilon
    \end{equation}
    holds for all $k\geq0$. We first prove that \eqref{gk} implies that
    \begin{equation}\label{ggk}
    \sum\limits_{k\in\mathcal{S}}\sqrt{\frac{\|g_k\|}{\sigma_k}}<+\infty,
    \end{equation}
    which further gives that $\sqrt{\|g_k\|/\sigma_k}\rightarrow 0$, and hence $\sigma_k\rightarrow \infty$ as $k\rightarrow\infty$.\\
    It follows from \eqref{pred}, \eqref{gk} and the construction of algorithm that
    \begin{align}\label{temp0}
    f(x_k)-f(x_{k+1})\geq&\eta_1[f(x_k)-m_k(s_k)]\nonumber\\
    \geq&\frac{\eta_1\epsilon}{12}\min\Bigg\{\frac{\epsilon}{\gamma_{\max}},\frac12\sqrt{\frac{\|g_k\|}{\sigma_k}}\Bigg\}, k\in\mathcal{S}.
    \end{align}
    Since the sequence $\{f(x_k)\}$ is monotonically decreasing and bounded below, it is convergent. As $k\rightarrow \infty$, the minimum on the right-hand side of \eqref{temp0} will be attained at $\sqrt{\|g_k\|}/(2\sqrt{\sigma_k})$ and the left-hand side of \eqref{temp0} converges to zero. Thus we obtain
    \begin{equation}\label{tem0p}
    f(x_k)-f(x_{k+1})\geq\frac{\eta_1\epsilon}{24}\sqrt{\frac{\|g_k\|}{\sigma_k}},
    \end{equation}
    for all $k\in\mathcal{S}$ sufficiently large.
    Summing up \eqref{tem0p} over all sufficiently large iterations provides that
    \begin{equation}\label{temp00}
    f(x_{k_0})-f(x_{j+1})=\sum\limits_{k=k_0,k\in\mathcal{S}}^j[f(x_k)-f(x_{k+1})]\geq\frac{\eta_1\epsilon}{24}\sum\limits_{k=k_0,k\in\mathcal{S}}^j\sqrt{\frac{\|g_k\|}{\sigma_k}}
    \end{equation}
    holds for some index $k_0$ sufficiently large and for any $j\in\mathcal{S}$, $j\geq k_0$. Since $\{f(x_{j+1})\}$ is convergent, letting $j\rightarrow\infty$ in \eqref{temp00} yields \eqref{ggk}.
    
    Now we show that the sequence of iterates $\{x_k\}$ is a Cauchy sequence. It follows from \eqref{ssk} and the construction of the algorithm that
    \begin{align*}
    \|x_{l+r}-x_l\|\leq&\sum\limits_{k=l}^{l+r-1}\|x_{k+1}-x_k\|=\sum\limits_{k=l,k\in\mathcal{S}}^{l+r-1}\|s_k\|\nonumber\\
    \leq&\sum\limits_{k=l,k\in\mathcal{S}}^{l+r-1}\sqrt{\frac{\|g_k\|}{\sigma_k}},
    \end{align*}
    for $l\geq0$ sufficiently large and any $r\geq0$, whose right-hand term tends to zero as $l\rightarrow\infty$ due to \eqref{ggk}. Thus $\{x_k\}$ is a Cauchy sequence, and there is some $x_*\in \mathbb{R}^n$ such that
    \begin{equation}\label{xk}
    x_k\rightarrow x_*, k\rightarrow\infty.
    \end{equation}
    Then all the conditions of Lemma \ref{lem3} hold with $\mathcal{I}:=\mathcal{S}$, and all $k\in\mathcal{S}$ sufficiently large are very successful. Now, if all $k$ sufficiently large belong to $\mathcal{S}$, then by \eqref{update}, $\sigma_{k+1}\leq\sigma_k$ for all $k$ sufficiently large, and so $\{\sigma_k\}$  is bounded above. This, however, contradicts $\sigma_k\rightarrow\infty$. Thus \eqref{gk} cannot hold.
    
    It remains to show that all sufficiently large iterations belong to $\mathcal{S}$. Conversely, let $\{k_i\}$ denote an infinite subsequence of very successful iterations such that $k_i-1$ is unsuccessful for all $i\geq0$. Then, $\sigma_{k_i}=c_1\sigma_{k_i-1}$ and $g_{k_i}=g_{k_i-1}$ for all $i$. Thus, we deduce that
    \begin{equation}\label{gk-1}
    \sqrt{\|g_{k_i-1}\|/\sigma_{k_i-1}}\rightarrow 0, i\rightarrow\infty.
    \end{equation}
    It follows from \eqref{gk}, \eqref{xk} and \eqref{gk-1} that \eqref{ass1}--\eqref{ass2} are satisfied with $\mathcal{I}:=\{k_i-1:i\geq0\}$, and so Lemma \ref{lem2} provides that $k_i-1$ is very successful for all $i$ sufficiently large. This contradicts our assumption that $k_i-1$ is unsuccessful for all $i$.
\end{proof}
Furthermore, if we assume that $g$ is uniformly continuous, the whole sequence of gradients $\{g_k\}$ converges to zero.
\begin{theorem}\label{th2}
    Let $f\in C^1$ and its gradient $g$ is uniformly continuous, $\{x_k\}$ is the sequence generated by Algorithm \ref{marc}. If $\{f(x_k)\}$ is bounded below, then
    \begin{equation}\label{theo2}
    \lim_{k\rightarrow\infty}\|g_k\|=0.
    \end{equation}
\end{theorem}
\begin{proof}
    Lemma \ref{lem3} implies that \eqref{theo2} holds if there are finitely many successful iterations. Now assume that there is an infinite subsequence $\{t_i\}\subseteq\mathcal{S}$ such that $\|g_{t_i}\|\geq2\epsilon$, for some $\epsilon>0$ and for all $i$.
    By Theorem \ref{theo1}, for each $t_i$, there exists a first successful iteration $l_i>t_i$ such that $\|g_{l_i}\|<\epsilon$. Thus $\{l_i\}\subseteq\mathcal{S}$, and we have
    \begin{equation}\label{gk2}
    \|g_{t_i}\|\geq\epsilon,
    \end{equation}
    for all $i$ and for all $k$ with $t_i\leq k<l_i$.
    Let $\mathcal{K}\triangleq\{k\in\mathcal{S}:t_i\leq k<l_i\}$, where $\{t_i\}$ and $\{l_i\}$ were defined above. Since $\mathcal{K}\subseteq\mathcal{S}$, it follows from \eqref{pred}, \eqref{gk} and the construction of algorithm that
    \begin{equation}\label{temp}
    f(x_k)-f(x_{k+1})\geq\frac{\eta_1\epsilon}{12}\min\Bigg\{\frac{\epsilon}{\gamma_{\max}},\frac12\sqrt{\frac{\|g_k\|}{\sigma_k}}\Bigg\}, k\in\mathcal{K}.
    \end{equation}
    Since $\{f(x_k)\}$ is monotonically decreasing and bounded from below, it is convergent, i.e., $f(x_k)-f(x_{k+1})\rightarrow 0$ as $k\rightarrow\infty$, thus
    \begin{equation*}
    \sqrt{\|g_k\|/\sigma_k}\rightarrow 0, k\rightarrow\infty, k\in\mathcal{K},
    \end{equation*}
    which, together with \eqref{ssk} and \eqref{temp}, implies for all $t_i\leq k<l_i$, $k\in\mathcal{S}$, $i$ sufficiently large,
    \begin{equation*}
    f(x_k)-f(x_{k+1})\geq\frac{\eta_1\epsilon}{24}\|s_k\|.
    \end{equation*}
    Summing up the above inequation over $k$ with $t_i\leq k<l_i$, we have
    \begin{equation*}
    \|x_{t_i}-x_{l_i}\|\leq\sum\limits_{k=t_i}^{l_i-1}\|x_{k+1}-x_k\|=\sum\limits_{k=t_i}^{l_i-1}\|s_k\|\leq\frac{24}{\eta_1\epsilon}[f(x_{t_i})-f(x_{l_i})],
    \end{equation*}
    for all $i$ sufficiently large. Since $\{f(x_k)\}$ is convergent, $\{f(x_{t_i})-f(x_{l_i})\}$ converges to zero as $i\rightarrow\infty$. Thus $\|x_{t_i}-x_{l_i}\|\rightarrow 0$ as $i\rightarrow\infty$, and hence $\|g_{t_i}-g_{l_i}\|\rightarrow 0$ since $g$ is uniformly continuous. Then we have reached a contradiction, since $\|g_{t_i}-g_{l_i}\|\geq\|g_{t_i}\|-\|g_{l_i}\|\geq\epsilon$ for all $i>0$.
\end{proof}

\textbf{Remark:} We could set $\gamma_{\min}=0$ in Algorithm 1, which means that we will ignore the second-order term when $\gamma_k<0$. In this case, the model
\[m_k(s)=f(x_k)+g_k^Ts+\frac{\sigma_k}3\|s\|^3\]
is just the linear model plus a term of cubic regularization, the minimizer of $m_k(s)$ is $s=-\frac{g_k}{\sqrt{\sigma_k\|g_k\|}}$, which can be reduced from (2) with $\gamma_k=0$, and $\|s_k\|=\sqrt{\|g_k\|/{\sigma_k}}$. Furthermore, 
\[f(x_k)-m_k(s_k)=\frac23\sqrt{\frac{\|g_k\|}{\sigma_k}}\|g_k\|.\]
Then the argument of Lemma \ref{lem1} still holds. Thus this change will not affect our above analysis.

\section{A Variant of MARC Algorithm}\label{sec4}
Recently, nonmonotone technique is always applied to develop more efficient optimization algorithms. Combining with the nonmonotone line search of Grippo et al. \cite{Grippo1986a}, Raydan \cite{Raydan1997} proposed a globally convergent Barzilai and Borwein gradient method for large scale unconstrained optimization problem, Deng et al. \cite{Deng1993} and Sun \cite{Sun2004} generalized the above nonmonotone rule into the trust region framework. Bianconcini and Sciandrone \cite{Bianconcini2016a} introduced the line search and nonmonotone techniques to the ARC algorithm. In this section, we will present a variant of MARC algorithm based on Hager and Zhang's nonmonotone technique \cite{Zhang2004a}.

To determine the step-size in the line search methods, Hager and Zhang \cite{Zhang2004a} considered the following type of nonmonotone line search rule
\begin{equation*}
f(x_k+\alpha_kd_k) \leq C_k+\sigma\alpha_k g_k^T d_k,
\end{equation*}
where $C_k$ is defined as
\begin{equation*}
C_k=\frac{\eta_{k-1}Q_{k-1}C_{k-1}+f(x_k)}{Q_k},
\end{equation*}
in which $Q_k=\eta_{k-1}Q_{k-1}+1$. The initial value $C_0=f_0$ and $Q_0=1$.

Replacing $f(x_k)$ with $C_k$ in \eqref{rhok}, then we state a variant of MARC algorithm in Algorithm \ref{vmarc}.

\begin{algorithm}[H]
    \caption{A Variant of MARC Algorithm (VMARC)}
    \label{vmarc}
    \begin{algorithmic}[1]
        \REQUIRE $x_0$, $\eta_2\geq\eta_1>0$, $c_1\geq1\geq c_2$, $\gamma_{\max}>\gamma_0>\gamma_{\min}>0$, $\sigma_0>0$, $C_0,Q_0$\\ For $k=0,1,\cdots$ until convergence,
        \STATE Compute $s_k$ by \eqref{sk}.
        \STATE Compute $f(x_k+s_k)$ and \[\rho_k=\frac{C_k-f(x_k+s_k)}{f(x_k)-m_k(s_k)};\]
        \STATE Set
        \begin{align*}
        x_{k+1}=\left\{\begin{array}{ll}x_k+s_k &\text{if}~\rho_k\geq \eta_1,\\x_k&\text{otherwise}.\end{array}\right.
        \end{align*}
        \STATE Set
        \begin{equation}\label{update1}
        \sigma_{k+1}=\left\{
        \begin{array}{lll}
        c_2\sigma_k, & \text{if}~\rho_k>\eta_2 & \text{[very successful iteration]},\\
        \sigma_k, & \text{if}~\eta_1\leq\rho_k\leq\eta_2 & \text{[successful iteration]},\\
        {c_1\sigma_k}, & \text{otherwise} & \text{[unsuccessful iteration]}.
        \end{array}\right.
        \end{equation}
        \STATE Compute $\gamma_{k+1}$ by \eqref{gamma}, \eqref{bb1} or \eqref{bb2}. Set $\gamma_{k+1}=\max\{\gamma_{\min},\min\{\gamma_{k+1},\gamma_{\max}\}\}$.
        \STATE Update $C_k$ and $Q_k$ when $k$ is a very successful or successful iteration.
    \end{algorithmic}
\end{algorithm}

Now we analyze the convergence of this variant of the MARC algorithm. Some of the following analysis is similar to that of MARC method in Sect. \ref{sec3}. For the completeness of the paper, we give a sketch of the proof.
\begin{lemma}\label{lem4}
    Let $\{x_k\}$ be the sequence generated by Algorithm 2. Then for any $k\geq0$, we have
    \begin{equation}\label{ck}
    f_{k+1}\leq C_{k+1}\leq C_k.
    \end{equation}
\end{lemma}
\begin{proof}
    Obviously, there is $f(x_{k+1})\leq C_k$. Since $C_{k+1}$ is a convex combination of $C_k$ and $f(x_{k+1})$, we have $f_{k+1}\leq C_{k+1}\leq C_k$ for each $k$.
\end{proof}

We note that the conclusions of Lemma \ref{lem2} and Lemma \ref{lem3} in Section 2 also hold in this case.

\begin{lemma}\label{lem5}
    Let $f\in C^1$. Suppose that $\epsilon>0$ and $\mathcal{I}$ is an infinite index set such that
    \begin{equation*}
    \|g_k\|\geq\epsilon, k\in \mathcal{I},
    \end{equation*}
    and
    \begin{equation*}
    \sqrt{\frac{\|g_k\|}{\sigma_k}}\rightarrow 0, k\rightarrow \infty, k\in \mathcal{I}.
    \end{equation*}
    Furthermore, if for some $x_*\in\mathbb{R}^n$, $x_k\rightarrow x_*, k\rightarrow \infty, k\in \mathcal{I}$, then each iteration $k\in \mathcal{I}$ that is sufficiently large is very successful, and
    \begin{equation*}
    \sigma_{k+1}\leq\sigma_k
    \end{equation*}
    holds for all $k\in \mathcal{I}$ sufficiently large.
\end{lemma}

\begin{proof}
    Firstly, we have that
    \begin{equation*}
    f(x_k+s_k)-m_k(s_k)\leq\Bigg\{\|g(\xi_k)-g_k\|+\frac{\gamma_k}2\sqrt{\frac{\|g_k\|}{\sigma_k}}\Bigg\}\cdot\sqrt{\frac{\|g_k\|}{\sigma_k}},
    \end{equation*}
    and
    \begin{equation*}
    f(x_k)-m_k(s_k)\geq\frac{\|g_k\|}{24}\sqrt{\frac{\|g_k\|}{\sigma_k}},
    \end{equation*}
    for all $k\in \mathcal{I}$ sufficiently large.
    Thus, by \eqref{diff}, we obtain that
    \begin{align}\label{rk1}
    r_k:=&f(x_k+s_k)-C_k-\eta_2[m_k(s_k)-f(x_k)]\nonumber\\
    =&f(x_k+s_k)-m_k(s_k)+(1-\eta_2)[m_k(s_k)-f(x_k)]+(f_k-C_k)\nonumber\\
    \leq&\sqrt{\frac{\|g_k\|}{\sigma_k}}\Bigg\{\|g(\xi_k)-g_k\|+\frac{\gamma_{\max}}2\sqrt{\frac{\|g_k\|}{\sigma_k}}-\frac{(1-\eta_2)\epsilon}{24}\Bigg\}+(f_k-C_k),
    \end{align}
    for all $k\in\mathcal{I}$ sufficiently large.
    
    Since $\xi_k$ belongs to the line segment $(x_k,x_k+s_k)$, we have $\|\xi_k-x_*\|\leq\|x_k-x_*\|+\|s_k\|$. In addition, $\|s_k\|\rightarrow 0$ and $\xi_k\rightarrow x_*$, as $k\rightarrow \infty$, $k\in \mathcal{I}$ and $g$ is continuous, we conclude that $\|g(\xi_k)-g_k\|\rightarrow 0$ as $k\rightarrow \infty$, $k\in \mathcal{I}$.
    
    Thus, the limit of $\sqrt{\|g_k\|/\sigma_k}$, \eqref{rk1} and \eqref{ck} imply that $r_k<0$, for all $k\in \mathcal{I}$ sufficiently large. Recalling that
    \begin{equation*}
    r_k<0 \Leftrightarrow \rho_k=\frac{f(x_k)-f(x_k+s_k)}{f(x_k)-m_k(s_k)}>\eta_2,
    \end{equation*}
    we conclude that the $k$-th iteration is very successful, and hence $\sigma_{k+1}\leq\sigma_k$.
\end{proof}

\begin{theorem}\label{theo3}
    Let $f\in C^1$. Suppose that $\{x_k\}$ is the sequence generated by Algorithm \ref{vmarc}. If $\{f(x_k)\}$ is bounded below, then
    \begin{equation}\label{them3}
    \liminf_{k\rightarrow\infty}\|g_k\|=0.
    \end{equation}
\end{theorem}
\begin{proof}
    By way of contradiction, we assume that there is some $\epsilon>0$ such that
    \begin{equation}\label{gk3}
    \|g_k\|\geq\epsilon
    \end{equation}
    holds for all $k\geq0$. It follows from $\rho_k\geq\eta_1$ and Lemma 2 that
    \begin{align*}
    f(x_{k+1})\leq C_k-\frac{\eta_1\epsilon}{6\sqrt{2}}\min\Bigg\{\frac{\epsilon}{\gamma_{\max}},\frac12\sqrt{\frac{\|g_k\|}{\sigma_k}}\Bigg\}, k\in\mathcal{S}.
    \end{align*}
    which, together with \eqref{gk3} and the definition of $C_k$, implies
    \begin{align}\label{tempp}
    C_k-C_{k+1}=&C_k-\frac{\eta_kQ_kC_k+f_{k+1}}{Q_{k+1}}\nonumber\\
    \geq&C_k-\frac{\eta_kQ_kC_k+C_k-\frac{\eta_1\epsilon}{6\sqrt{2}}\min\Big\{\frac{\epsilon}{\gamma_{\max}},\frac12\sqrt{\frac{\|g_k\|}{\sigma_k}}\Big\}}{Q_{k+1}}\nonumber\\
    =&\frac{\eta_1\epsilon}{12Q_{k+1}}\min\Bigg\{\frac{\epsilon}{\gamma_{\max}},\frac12\sqrt{\frac{\|g_k\|}{\sigma_k}}\Bigg\}.
    \end{align}
    
    Since $\{f(x_k)\}$ is bounded below and $f_k\leq C_k$, so $\{C_k\}$ is bounded below. Also, \eqref{tempp} implies that $\{C_k\}$ is monotonically decreasing, and hence it is convergent. Therefore the left-hand side of \eqref{tempp} converges to zero as $k\rightarrow\infty$, thus,
    \begin{equation*}
    \sqrt{\|g_k\|/\sigma_k}\rightarrow 0, k\rightarrow\infty.
    \end{equation*}
    The rest analysis is similar to that of Theorem \ref{theo1}, so we omit it.
\end{proof}

Furthermore, if we assume that $g$ is uniformly continuous, we have the strong convergence of the VMARC method.
\begin{theorem}
    Let $f\in C^1$ and its gradient $g$ is uniformly continuous, $\{x_k\}$ is the sequence generated by Algorithm \ref{vmarc}. If $\{f(x_k)\}$ is bounded below, then
    \begin{equation*}
    \lim_{k\rightarrow\infty}\|g_k\|=0.
    \end{equation*}
\end{theorem}
\begin{proof}
    The proof is similar to that of Theorem \ref{th2}, so we omit it.
\end{proof}

\section{Numerical Experiments}\label{sec5}
In this section, through testing a set of 54 unconstrained optimization problems from CUTEst library \cite{Gould2014a}, we analysis the effectiveness of Algorithm VMARC and compare it with the TRMSM algorithm of Zhou et al. \cite{Zhou2016}, the standard trust region method and adaptive regularization algorithm with cubics. All the codes were written in Fortran 77 with double precision and run on a personal desktop computer with 3.60 GHz CPU, 4.0 GB memory and 64-bit Ubuntu operation system.

In all our tests, the parameters in Algorithm VMARC are as follows: $\sigma_0=1$, $\eta_1=0.1$, $\eta_2=0.75$, $c_1=5$, $c_2=0.2$, $\gamma_{\max}=10^6$, $\psi_k=0.2$ and $\eta_k=0.7$. We call the VMARC type method employing \eqref{gamma}, \eqref{bb1}, \eqref{bb2} as the $\gamma_k$ parameter the MARC1 method, MARC2 method and MARC3 method, respectively. The corresponding TRMSM type methods are the TRMSM1 method, TRMSM2 method and TRMSM3 method.

Our experiments contains two parts. In the first part, we give some numerical comparisons between TRMSM3, MARC3, standard TR and ARC methods. During this work, the TRU and ARC packages (Fortran implementation of the standard TR and ARC methods) of GALAHAD \cite{Gould2003a} are used and we change the stopping criteria of MARC3 and TEMSM3 methods as TRU and ARC packages, that is, we stop the iteration when \[\|g_k\|\leq\max\{10^{-5},10^{-9}\|g_0\|\}.\]
The detailed numerical results\footnote{The complete data can be downloaded from https://pan.baidu.com/s/1CWF2WoQSdSksIbwi-4AXsg.} are included in Table 1 with the the number of iterations (iter), the number of function-gradient evaluations (nf) and the cost time (cpu). The sign `-' denotes a failure of the algorithm.

\begin{figure}[htbp]
    \centering
    \includegraphics[scale=0.39]{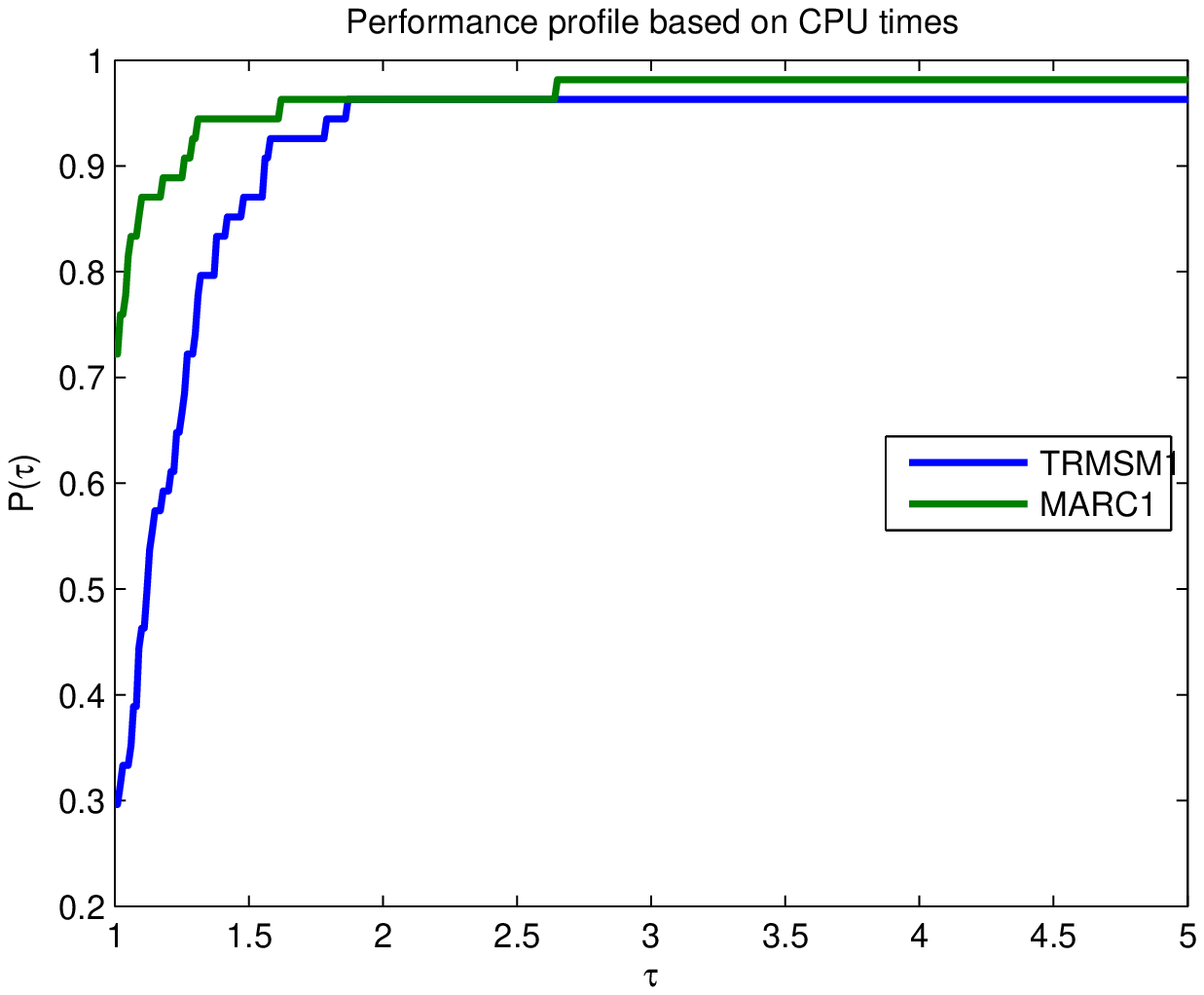}
    \includegraphics[scale=0.39]{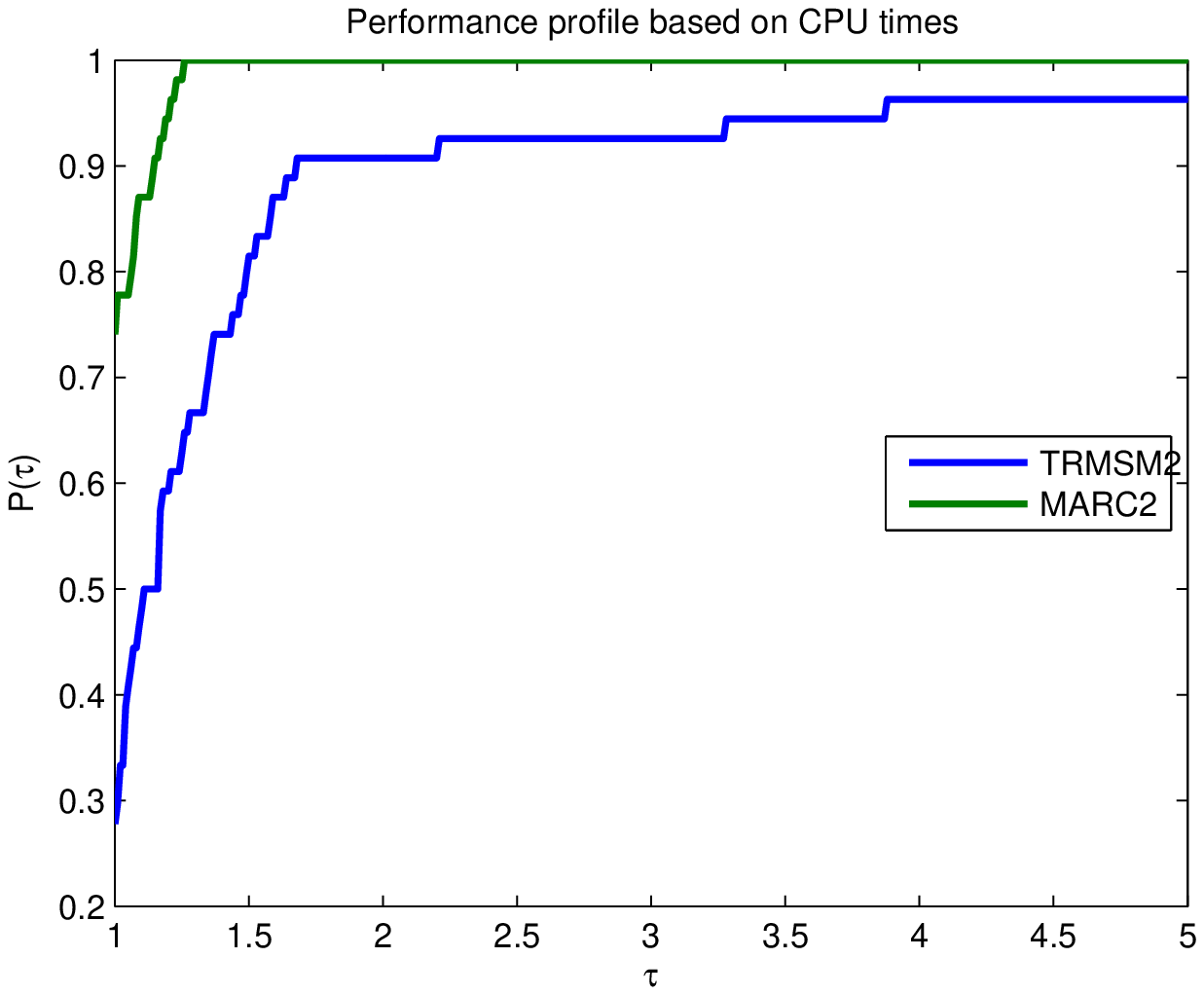}
    \includegraphics[scale=0.39]{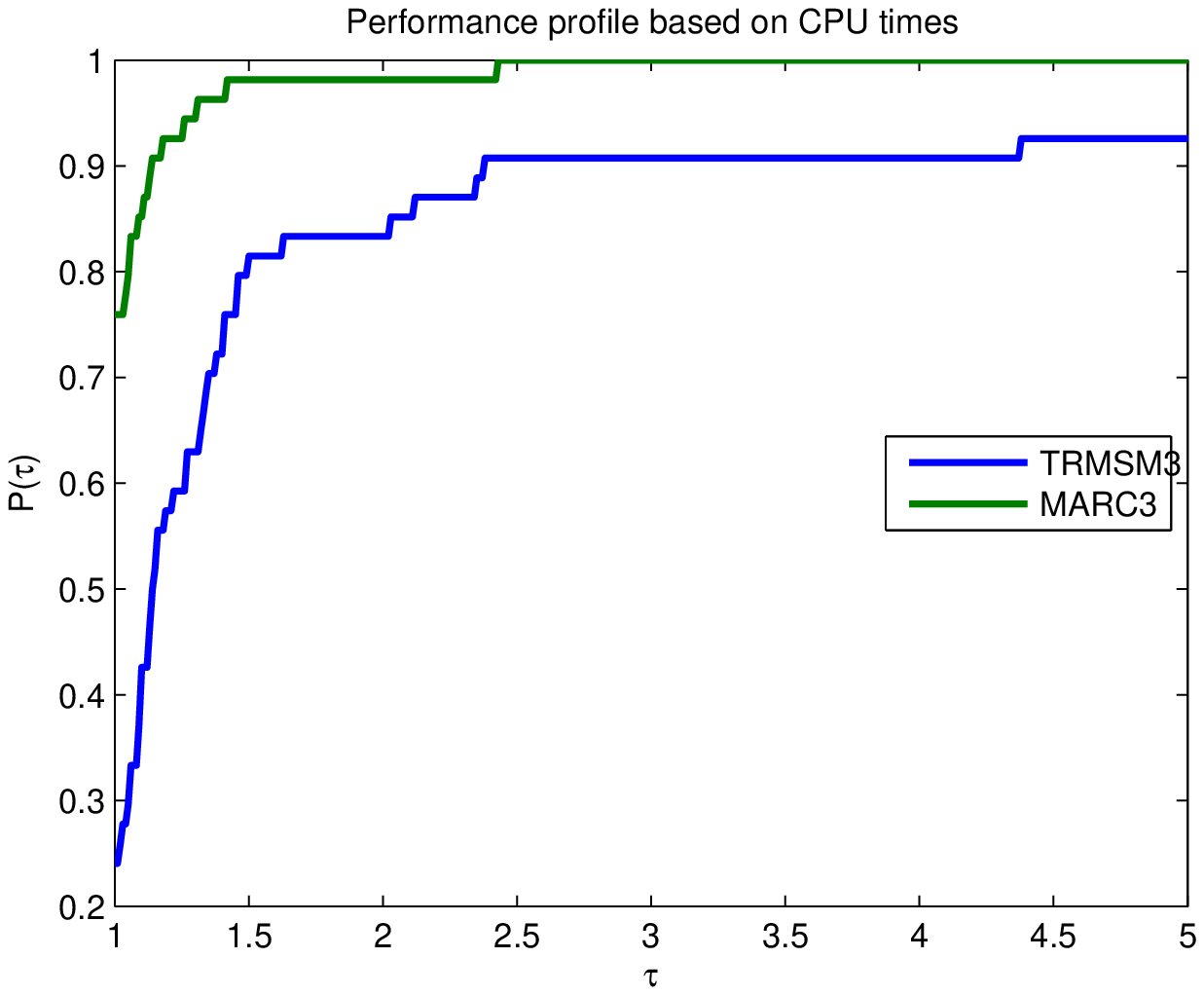}
    \includegraphics[scale=0.39]{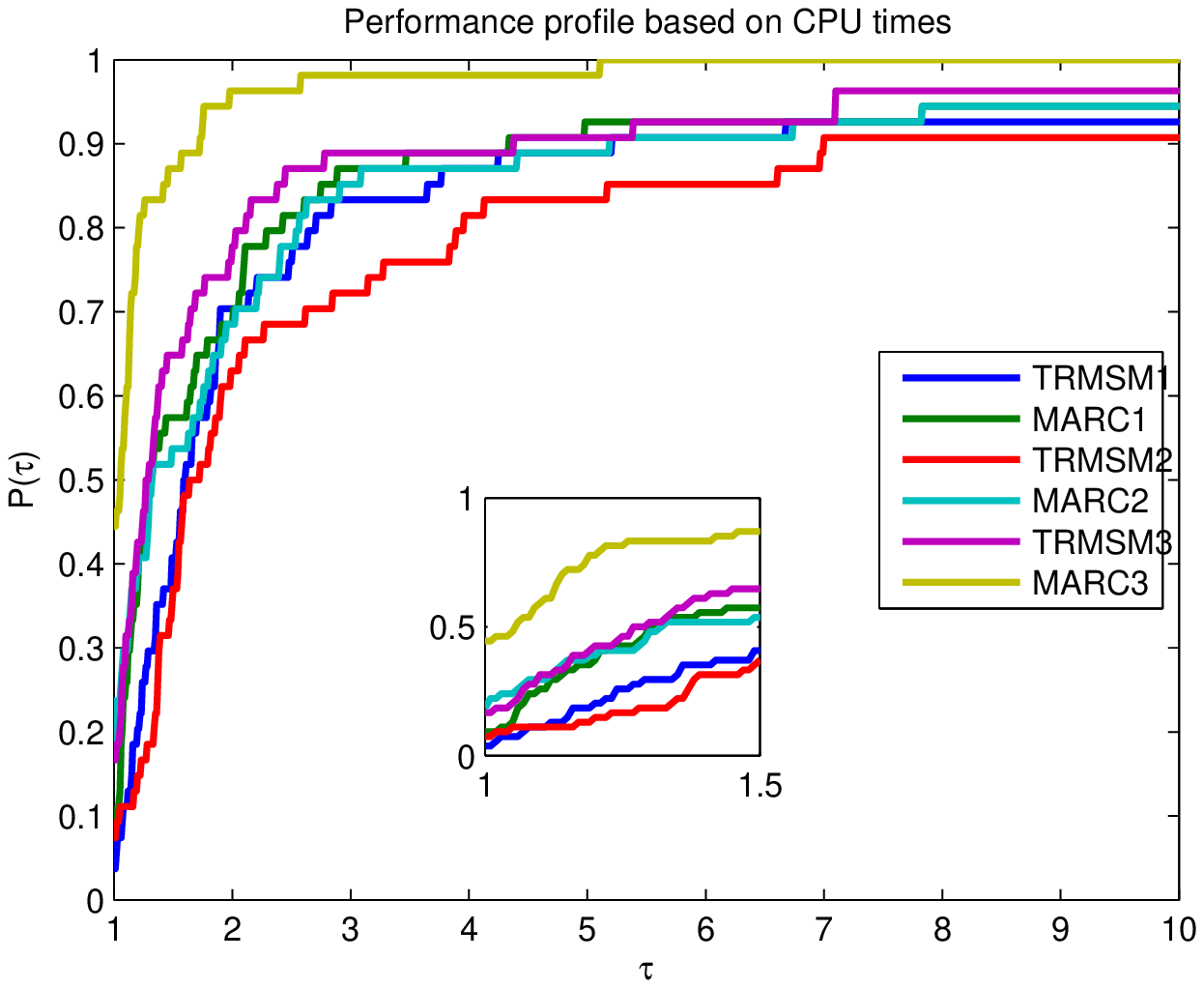}
    \caption{Performance profiles based on CPU time}
    \label{fig1}
\end{figure}
For each problem, some more information are also given in Table 1, $r$ is the ratio of the cost time of computing the exact Hessians once and cost time of one function-gradient calculation, the 4th column Time gives the cost time of 20 function-gradient calculations. For different $r$, when the exact Hessian is used, the TRU and ARC method always need less iterations but more cost time than the TRMSM3 and MARC3 methods. Firstly, when $r$ is very large, such as 3641 for problem VARDIM, calculating the exact Hessians even once could cost much computing resource. On the other hand, even if $r$ is small, another inner iterative algorithm for solving trust region or cubic model subproblems is still needed and this step may cost much time. But for the TRMSM3 or MARC3 method, we only use the function value and the gradient information, and at each step, $O(n)$ operations are needed to form $\gamma_k$.

Next, we mainly compare MARC with TRMSM methods when employing different $\gamma_k$. We terminate the iteration when
\begin{equation*}
\|g_k\|_\infty\leq10^{-6}(1+|f(x_k)|).
\end{equation*}
In addition, we also stop the algorithm if the number of iterations exceeds 5000. The detailed numerical results are given in Tables \ref{tab2}--\ref{tab3}. Dolan and Mor\'{e}'s \cite{Dolan2002} performance profiles are used to display the behaviors of these methods.

Figure \ref{fig1} shows comparison of different methods based on cpu time. By using the same $\gamma_k$ parameter, the VMARC algorithm obviously outperform the TRMSM algorithm. Among these six methods, the MARC3 method has the best numerical performance, which can solve about 45\% of the tested problems in the shortest time. Similar conclusions can be draw from Figure \ref{fig2}, which show the performance profiles based on iterations and function evaluations.

\begin{figure}[htbp]
    \centering
    \includegraphics[scale=0.39]{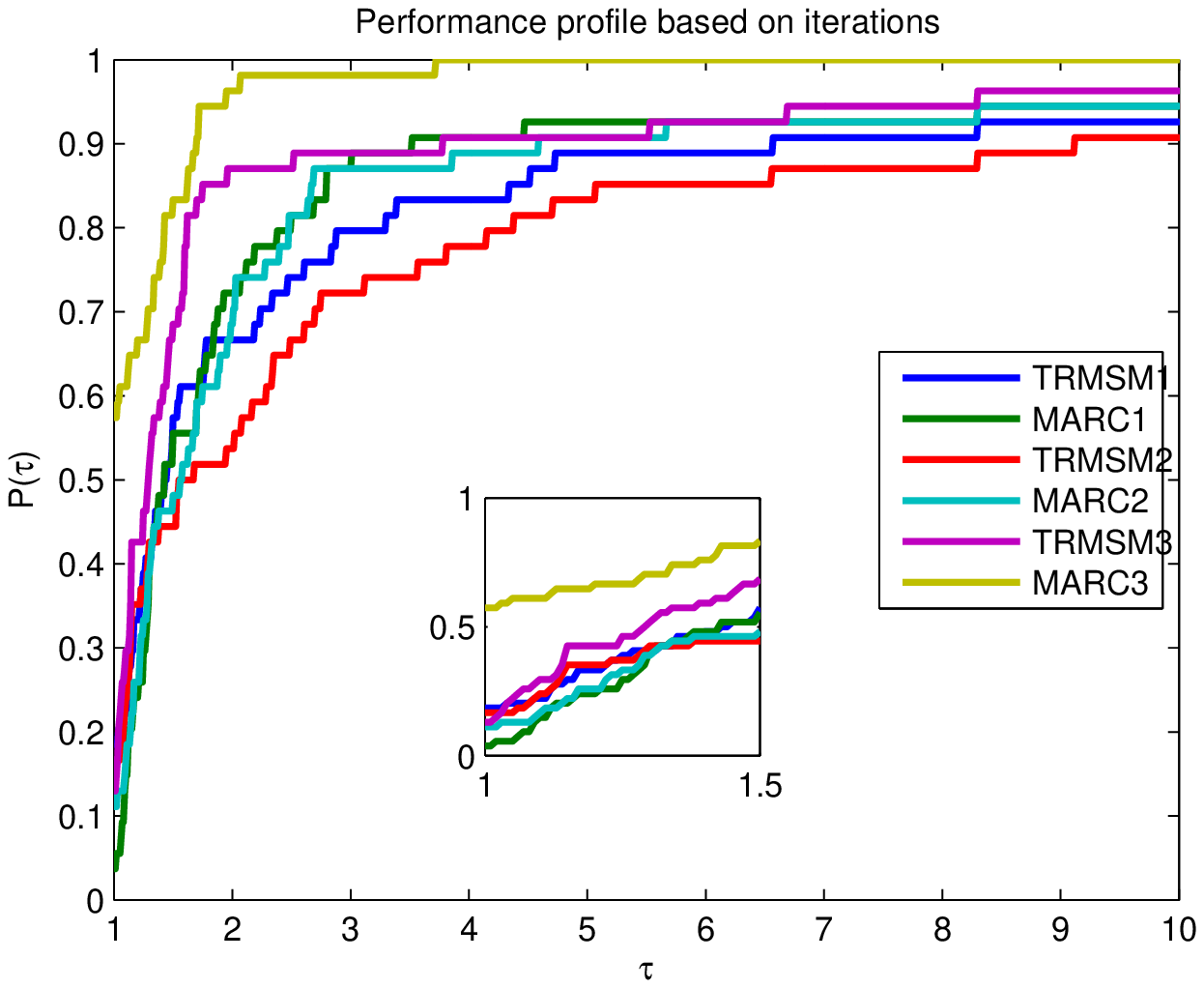}
    \includegraphics[scale=0.39]{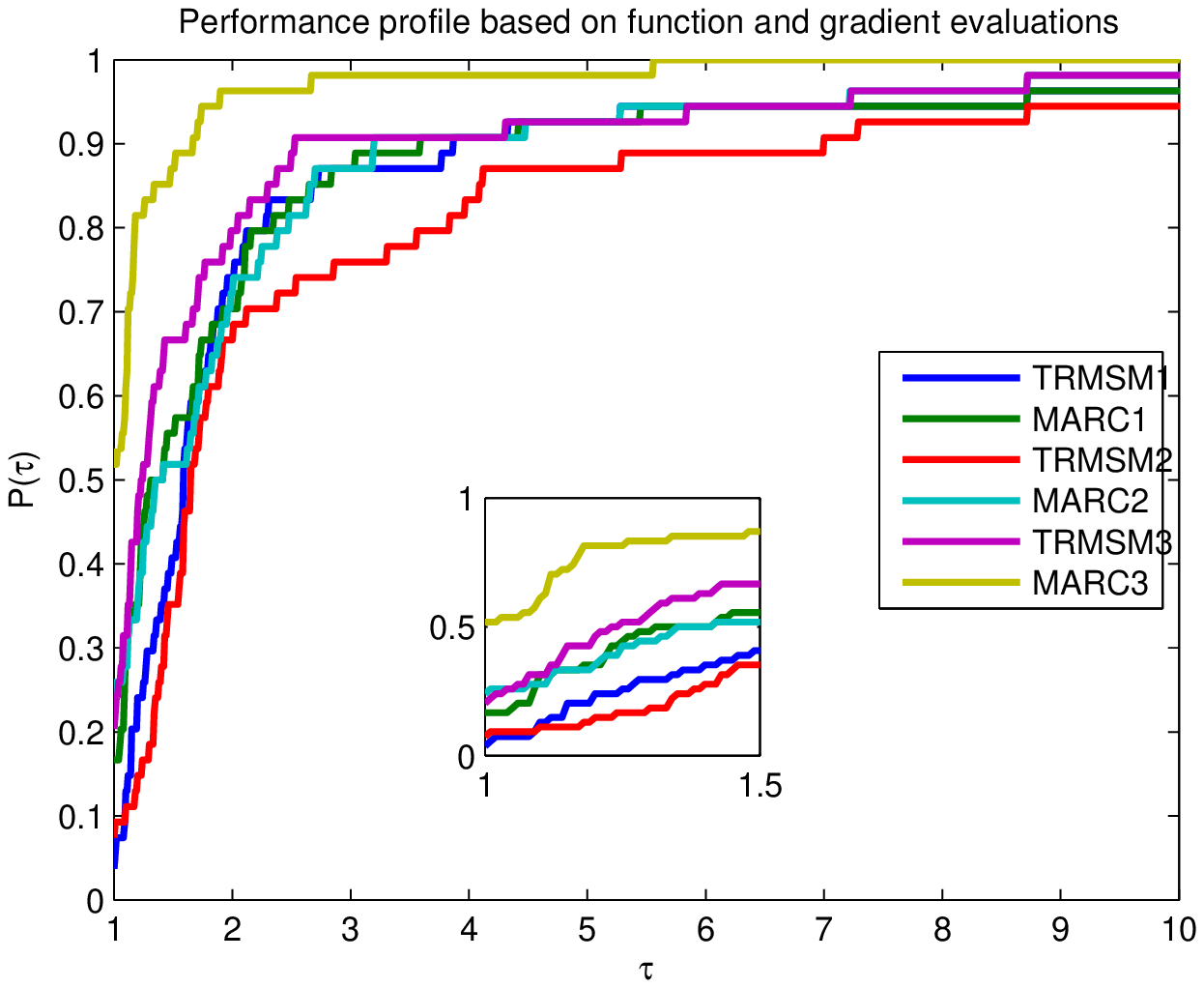}
    \caption{Performance profiles based on iterations and function evaluations}
    \label{fig2}
\end{figure}

\section{Conclusion}\label{sec6}
In this paper, combing with the Barzilai-Borwein step size, we present a modified adaptive cubic regularization method for large-scale unconstrained optimization problem. With nonmonotone technique, a variant of the modified method is also proposed. Convergence of the method is analyzed under some reasonable assumptions. Numerical experiments are performed to show the effectiveness of the new methods.

\begin{sidewaystable}[ht]
    \centering
    \caption{Numerical comparisons between ARC/TR and TRMSM/MARC methods}
    \vskip2mm
    \newsavebox{\tableboxaab}
    \begin{lrbox}{\tableboxaab}
        \begin{tabular}{lrrlrrrrrrrrrr}
            \toprule        
            \multicolumn{1}{c}{\multirow{2}[0]{*}{Problem}} & \multicolumn{1}{c}{\multirow{2}[0]{*}{Dim}} & \multicolumn{1}{c}{\multirow{2}[0]{*}{$r$}} & \multicolumn{1}{c}{\multirow{2}[0]{*}{Time}} & \multicolumn{2}{c}{ARC} & \multicolumn{2}{c}{TRU} & \multicolumn{3}{c}{TRMSM3} & \multicolumn{3}{c}{MARC3} \\
            \cmidrule(lr){5-6}\cmidrule(lr){7-8}\cmidrule(lr){9-11}\cmidrule(lr){12-14}
            &       &       &       & \multicolumn{1}{l}{iter} & \multicolumn{1}{l}{cpu} & \multicolumn{1}{l}{iter} & \multicolumn{1}{l}{cpu} & \multicolumn{1}{l}{iter} & \multicolumn{1}{l}{ng} & \multicolumn{1}{l}{cpu} & \multicolumn{1}{l}{iter} & \multicolumn{1}{l}{ng} & \multicolumn{1}{l}{cpu} \\
            \midrule
            POWER & 5000  & 8328.6 & 2.67-03 & 26    & 27.497 & 33    & 32.747 & 55    & 125   & 0.023 & 92    & 170   & 0.032 \\
            VARDIM & 2000  & 3641.1 & 1.59-03 & 27    & 6.873 & 719   & 182.221 & 52    & 157   & 0.015 & 5     & 36    & 0.003 \\
            BROWNAL & 200   & 586.47 & 1.49-02 & 7     & 0.349 & 7     & 0.315 & 6     & 23    & 0.004 & 4     & 11    & 0.002 \\
            PENALTY1 & 1000  & 456.2 & 1.02-03 & 28    & 0.759 & 37    & 0.933 & 12    & 46    & 0.003 & 6     & 14    & 0.001 \\
            FMINSURF & 1024  & 246.9 & 2.41-03 & 373   & 32.209 & 117   & 8.416 & 1172  & 1799  & 0.242 & 821   & 1548  & 0.202 \\
            SPARSQUR & 5000  & 218.2 & 1.07-02 & 22    & 0.205 & 20    & 0.153 & 22    & 41    & 0.024 & 17    & 22    & 0.013 \\
            DQRTIC & 2000  & 180.7 & 3.11-03 & 28    & 0.009 & 36    & 0.012 & 39    & 82    & 0.007 & 31    & 47    & 0.004 \\
            SROSENBR & 5000  & 140.9 & 4.05-03 & 12    & 0.016 & 10    & 0.012 & 27    & 59    & 0.015 & 29    & 49    & 0.013 \\
            NONDIA & 5000  & 102.1 & 5.61-03 & 29    & 0.033 & 24    & 0.03  & 4     & 25    & 0.008 & 5     & 14    & 0.004 \\
            WOODS & 4000  & 87.8  & 4.23-03 & 56    & 0.062 & 67    & 0.077 & 195   & 329   & 0.089 & 453   & 860   & 0.224 \\
            DQDRTIC & 5000  & 73.1  & 7.72-03 & 9     & 0.013 & 2     & 0.004 & 20    & 29    & 0.014 & 35    & 36    & 0.017 \\
            ARWHEAD & 5000  & 70.5  & 8.26-03 & 10    & 0.015 & 6     & 0.009 & 13    & 29    & 0.013 & 14    & 21    & 0.01 \\
            ENGVAL1 & 5000  & 65.8  & 8.70-03 & 12    & 0.035 & 9     & 0.026 & 25    & 34    & 0.018 & 22    & 23    & 0.012 \\
            SINQUAD & 5000  & 63.1  & 9.38-03 & 34    & 0.089 & 45    & 0.118 &    -   &   -    &    -   & 305   & 591   & 0.428 \\
            CRAGGLVY & 5000  & 40.8  & 1.44-02 & 17    & 0.102 & 15    & 0.082 & 144   & 246   & 0.234 & 138   & 263   & 0.252 \\
            DIXMAANJ & 3000  & 36.7  & 5.79-03 & 50    & 0.273 & 76    & 0.28  & 683   & 1075  & 0.355 & 444   & 839   & 0.277 \\
            DIXMAANJ& 9000  & 275.1 & 1.72-02 & 68    & 1.086 & 80    & 0.873 & 531   & 840   & 0.826 & 410   & 776   & 0.752 \\
            DIXMAAND & 3000  & 35.7  & 5.91-03 & 16    & 0.082 & 17    & 0.056 & 10    & 17    & 0.005 & 11    & 12    & 0.004 \\
            DIXMAAND& 9000  & 108.3 & 1.74-02 & 17    & 0.326 & 43    & 0.275 & 9     & 16    & 0.016 & 12    & 13    & 0.014 \\
            MOREBV & 5000  & 31.2  & 1.90-02 & 2     & 0.286 & 1     & 0.312 & 134   & 213   & 0.083 & 129   & 247   & 0.091 \\
            COSINE & 2000  & 24.2  & 3.69-03 & 42    & 0.053 &   -    &   -    & 11    & 13    & 0.004 & 15    & 17    & 0.005 \\
            EDENSCH & 2000  & 21.3  & 4.42-03 & 17    & 0.027 & 16    & 0.024 & 27    & 36    & 0.009 & 29    & 30    & 0.008 \\
            BROYDN7D & 5000  & 20.2  & 2.92-02 & 767   & 13.791 & 679   & 6.553 & 2325  & 3563  & 6.938 & 2039  & 3688  & 7.159 \\
            EXTROSNB & 5000  & 13.3  & 1.01-03 & 685   & 16.134 & 703   & 1.788 & 35    & 45    & 0.015 & 76    & 143   & 0.045 \\
            EG2   & 1000  & 6.9   & 1.96-03 & 7     & 0.003 & 3     & 0.001 & 5     & 16    & 0.002 & 6     & 11    & 0.001 \\
            SENSORS & 1000  & 1.7   & 5.24-00 & 683   & 316.162 & 33    & 36.078 & 42    & 66    & 22.958 & 29    & 33    & 9.99 \\
            \bottomrule
        \end{tabular}%
    \end{lrbox}
    \scalebox{1}{\usebox{\tableboxaab}}
    \label{tab1}%
\end{sidewaystable}%

\begin{sidewaystable}[ht]
    \centering
    \caption{Numerical results of VMARC and TRMSM methods}
    \vskip2mm
    \newsavebox{\tableboxaac}
    \begin{lrbox}{\tableboxaac}
        \begin{tabular}{lrrrrrrrrrrrrrrrrrrr}
            \toprule
            \multicolumn{1}{c}{\multirow{2}[0]{*}{Problem}} & \multicolumn{1}{c}{\multirow{2}[0]{*}{n}} & \multicolumn{3}{c}{TRMSM1} & \multicolumn{3}{c}{MARC1} & \multicolumn{3}{c}{TRMSM2} & \multicolumn{3}{c}{MARC2} & \multicolumn{3}{c}{TRMSM3} & \multicolumn{3}{c}{MARC3} \\
            \cmidrule(lr){3-5}\cmidrule(lr){6-8}\cmidrule(lr){9-11}\cmidrule(lr){12-14}\cmidrule(lr){15-17}\cmidrule(lr){18-20}
            &       & \multicolumn{1}{c}{iter} & \multicolumn{1}{c}{nf} & \multicolumn{1}{c}{cpu} & \multicolumn{1}{c}{iter} & \multicolumn{1}{c}{nf} & \multicolumn{1}{c}{cpu} & \multicolumn{1}{c}{iter} & \multicolumn{1}{c}{nf} & \multicolumn{1}{c}{cpu} & \multicolumn{1}{c}{iter} & \multicolumn{1}{c}{nf} & \multicolumn{1}{c}{cpu} & \multicolumn{1}{c}{iter} & \multicolumn{1}{c}{nf} & \multicolumn{1}{c}{cpu} & \multicolumn{1}{c}{iter} & \multicolumn{1}{c}{nf} & \multicolumn{1}{c}{cpu} \\
            \midrule
            ARWHEAD & 10000 & 10    & 27    & 0.0250  & 10    & 18    & 0.0170  & 10    & 27    & 0.0264  & 9     & 17    & 0.0161  & 13    & 30    & 0.0285  & 12    & 20    & 0.0197  \\
            BDQRTIC & 2000  & 3832  & 5964  & 1.3863  & 2469  & 4801  & 1.0979  & 3366  & 5242  & 1.1989  & 2188  & 4213  & 0.9635  & 1007  & 1584  & 0.3808  & 884   & 1691  & 0.4010  \\
            BOX   & 10000 & 3015  & 4689  & 6.1013  & 2009  & 3891  & 4.9861  & 2770  & 4305  & 5.5816  & 1793  & 3462  & 4.4329  & 698   & 1085  & 1.4385  & 668   & 1278  & 1.6854  \\
            BROWNAL & 400   & 8     & 27    & 0.0200  & 10    & 17    & 0.0129  & 8     & 27    & 0.0198  & 10    & 17    & 0.0129  & 10    & 29    & 0.0212  & 12    & 19    & 0.0146  \\
            BROYDN7D & 5000  & 2448  & 3752  & 7.1970  & 2178  & 3961  & 7.5734  & 2152  & 3272  & 6.2765  & 2195  & 3949  & 7.5585  & 2252  & 3455  & 6.6699  & 2007  & 3656  & 7.0330  \\
            BRYBND & 10000 & 36    & 48    & 0.0695  & 39    & 44    & 0.0641  & 41    & 60    & 0.0862  & 47    & 56    & 0.0808  & 41    & 60    & 0.0877  & 37    & 42    & 0.0625  \\
            CHAINWOO & 4000  & 2418  & 3739  & 1.6920  & 1781  & 3402  & 1.5169  & 2543  & 3923  & 1.7780  & 1901  & 3602  & 1.6060  &   -    &   -    &   -    & 714   & 1372  & 0.6252  \\
            CHNROSNB & 50    & 4153  & 6365  & 0.0286  & 3289  & 6174  & 0.0269  & 4154  & 6387  & 0.0282  & 3199  & 5986  & 0.0272  & 2339  & 3626  & 0.0169  & 1592  & 3015  & 0.0134  \\
            COSINE & 1000  & 8     & 10    & 0.0013  & 9     & 11    & 0.0014  & 7     & 9     & 0.0011  & 8     & 10    & 0.0012  & 8     & 10    & 0.0018  & 10    & 12    & 0.0016  \\
            CRAGGLVY & 10000 & 131   & 230   & 0.4323  & 153   & 296   & 0.5551  & 593   & 935   & 1.7850  & 127   & 244   & 0.4611  & 151   & 261   & 0.4977  & 117   & 227   & 0.4377  \\
            CURLY10 & 5000  & 1201  & 1858  & 0.8318  & 945   & 1770  & 0.7834  & 1208  & 1860  & 0.8360  & 857   & 1622  & 0.7156  & 711   & 1096  & 0.5082  & 515   & 969   & 0.4398  \\
            DIXMAANA & 9000  & 7     & 11    & 0.0111  & 10    & 11    & 0.0109  & 7     & 12    & 0.0112  & 9     & 10    & 0.0096  & 8     & 12    & 0.0119  & 10    & 11    & 0.0114  \\
            DIXMAANB & 9000  & 7     & 12    & 0.0121  & 10    & 11    & 0.0109  & 8     & 14    & 0.0130  & 9     & 10    & 0.0098  & 8     & 13    & 0.0127  & 9     & 10    & 0.0104  \\
            DIXMAANC & 9000  & 7     & 13    & 0.0126  & 12    & 13    & 0.0128  & 9     & 16    & 0.0150  & 11    & 12    & 0.0117  & 9     & 15    & 0.0146  & 12    & 13    & 0.0134  \\
            DIXMAAND & 9000  & 8     & 15    & 0.0144  & 12    & 13    & 0.0128  & 11    & 19    & 0.0180  & 11    & 12    & 0.0120  & 10    & 17    & 0.0164  & 13    & 14    & 0.0144  \\
            DIXMAANE & 9000  & 4155  & 6379  & 6.3765  & 3932  & 7401  & 6.9272  & 3855  & 5922  & 6.1231  & 3757  & 7040  & 6.5577  & 2448  & 3778  & 3.6779  & 1858  & 3517  & 3.3661  \\
            DIXMAANF & 9000  & 4567  & 6988  & 7.1188  & 3718  & 6922  & 6.5570  & 4322  & 6629  & 6.5669  & 3434  & 6486  & 6.0325  & 2718  & 4203  & 4.0829  & 1387  & 2618  & 2.5130  \\
            DIXMAANG & 9000  & 4504  & 6907  & 7.1448  & 3770  & 7073  & 6.5246  & 3715  & 5689  & 5.8587  & 3801  & 7159  & 6.8260  & 2527  & 3885  & 3.8372  & 1589  & 3019  & 2.8511  \\
            DIXMAANH & 9000  & 4500  & 6904  & 7.6149  & 3414  & 6453  & 6.0810  & 3891  & 5982  & 6.5304  & 3560  & 6701  & 6.4195  & 1993  & 3094  & 3.0510  & 1564  & 2989  & 2.8864  \\
            DIXMAANJ & 9000  & 1925  & 2966  & 3.0635  & 1343  & 2548  & 2.3553  & 1685  & 2604  & 2.5713  & 1461  & 2752  & 2.5503  & 1244  & 1928  & 1.8484  & 780   & 1470  & 1.3879  \\
            DIXMAANL & 9000  & 1053  & 1633  & 1.5780  & 817   & 1560  & 1.4602  & 1100  & 1726  & 1.6520  & 940   & 1771  & 1.6585  & 757   & 1188  & 1.1570  & 482   & 908   & 0.8718  \\
            DQDRTIC & 10000 & 27    & 36    & 0.0324  & 36    & 37    & 0.0339  & 27    & 36    & 0.0319  & 36    & 37    & 0.0338  & 24    & 33    & 0.0304  & 32    & 33    & 0.0317  \\
            DQRTIC & 2000  & 55    & 90    & 0.0078  & 51    & 65    & 0.0059  & 45    & 79    & 0.0068  & 41    & 51    & 0.0046  & 58    & 101   & 0.0092  & 58    & 85    & 0.0080  \\
            EDENSCH & 5000  & 21    & 30    & 0.0197  & 19    & 25    & 0.0156  & 18    & 31    & 0.0191  & 23    & 28    & 0.0174  & 19    & 28    & 0.0186  & 29    & 38    & 0.0243  \\
            EG2   & 1000  & 3     & 14    & 0.0017  & 4     & 9     & 0.0011  & 3     & 14    & 0.0016  & 4     & 9     & 0.0010  & 4     & 15    & 0.0017  & 5     & 10    & 0.0012  \\
            ENGVAL1 & 10000 & 14    & 23    & 0.0219  & 17    & 18    & 0.0178  & 10    & 18    & 0.0170  & 17    & 18    & 0.0179  & 16    & 25    & 0.0246  & 17    & 18    & 0.0187  \\
            EXTROSNB & 5000  & 69    & 101   & 0.0314  & 52    & 54    & 0.0176  &  -     &  -     &  -     & 91    & 172   & 0.0510  & 45    & 55    & 0.0186  & 77    & 144   & 0.0452  \\                  
            \bottomrule
        \end{tabular}%
    \end{lrbox}
    \scalebox{0.85}{\usebox{\tableboxaac}}
    \label{tab2}
\end{sidewaystable}%

\begin{sidewaystable}[ht]
    \centering
    \caption{Numerical results of VMARC and TRMSM methods (Continued)}
    \vskip2mm
    \newsavebox{\tableboxaad}
    \begin{lrbox}{\tableboxaad}
        \begin{tabular}{lrrrrrrrrrrrrrrrrrrr}
            \toprule
            \multicolumn{1}{c}{\multirow{2}[0]{*}{Problem}} & \multicolumn{1}{c}{\multirow{2}[0]{*}{n}} & \multicolumn{3}{c}{TRMSM1} & \multicolumn{3}{c}{MARC1} & \multicolumn{3}{c}{TRMSM2} & \multicolumn{3}{c}{MARC2} & \multicolumn{3}{c}{TRMSM3} & \multicolumn{3}{c}{MARC3} \\
            \cmidrule(lr){3-5}\cmidrule(lr){6-8}\cmidrule(lr){9-11}\cmidrule(lr){12-14}\cmidrule(lr){15-17}\cmidrule(lr){18-20}
            &       & \multicolumn{1}{c}{iter} & \multicolumn{1}{c}{nf} & \multicolumn{1}{c}{cpu} & \multicolumn{1}{c}{iter} & \multicolumn{1}{c}{nf} & \multicolumn{1}{c}{cpu} & \multicolumn{1}{c}{iter} & \multicolumn{1}{c}{nf} & \multicolumn{1}{c}{cpu} & \multicolumn{1}{c}{iter} & \multicolumn{1}{c}{nf} & \multicolumn{1}{c}{cpu} & \multicolumn{1}{c}{iter} & \multicolumn{1}{c}{nf} & \multicolumn{1}{c}{cpu} & \multicolumn{1}{c}{iter} & \multicolumn{1}{c}{nf} & \multicolumn{1}{c}{cpu} \\
            \midrule
            FMINSURF & 5625  &   -    &   -    &   -    & 4877  & 9110  & 6.6111  &   -    &    -   &   -    & 4299  & 8054  & 5.7941  & 3324  & 5094  & 3.7929  & 2540  & 4745  & 3.4735  \\
            FLETCBV3 & 10000 & 9     & 10    & 0.0166  & 10    & 11    & 0.0172  & 9     & 10    & 0.0163  & 11    & 12    & 0.0188  & 9     & 10    & 0.0171  & 10    & 11    & 0.0177  \\
            FREUROTH & 5000  & 1195  & 1849  & 1.1744  & 943   & 1768  & 1.1088  & 465   & 735   & 0.4651  & 289   & 554   & 0.3447  & 282   & 452   & 0.2908  & 51    & 105   & 0.0665  \\
            HILBERTA & 50    & 3257  & 4964  & 0.2494  & 2214  & 4156  & 0.2076  & 3253  & 4970  & 0.2476  & 2274  & 4211  & 0.2110  & 1247  & 1923  & 0.0968  & 496   & 941   & 0.0479  \\
            HILBERTB & 50    & 8     & 12    & 0.0007  & 8     & 9     & 0.0005  & 8     & 12    & 0.0007  & 8     & 9     & 0.0005  & 8     & 12    & 0.0007  & 8     & 9     & 0.0005  \\
            GENROSE & 500   &   -    &    -   &   -    & 4960  & 9190  & 0.3893  &    -   &    -   &   -    &  -     &   -    &    -   & 3969  & 6062  & 0.2706  & 3625  & 6721  & 0.2996  \\
            INDEF & 5000  &   -    &     -  &   -    &  -     &   -    &  -     &   -    &    -   &     -  &  -     &  -     &  -     &  -     &  -     &    -   & 3179  & 5816  & 3.4597  \\
            LIARWHD & 1000  & 2950  & 4619  & 0.3932  & 1744  & 3393  & 0.2869  & 2935  & 4581  & 0.4010  & 1644  & 3223  & 0.2738  & 905   & 1433  & 0.1316  & 624   & 1195  & 0.1044  \\
            MOREBV & 5000  & 174   & 269   & 0.0985  & 156   & 301   & 0.1079  & 197   & 309   & 0.1126  & 157   & 302   & 0.1073  & 134   & 213   & 0.0810  & 129   & 247   & 0.0912  \\
            NCB20 & 1010  & 1526  & 2326  & 3.5559  & 1626  & 3038  & 4.6309  & 1688  & 2591  & 3.9688  & 1725  & 3190  & 4.8743  & 926   & 1442  & 2.2099  & 868   & 1634  & 2.5030  \\
            NCB20B & 2000  & 90    & 162   & 0.5026  & 62    & 125   & 0.3866  & 74    & 137   & 0.4245  & 64    & 130   & 0.4018  & 58    & 102   & 0.3162  & 65    & 128   & 0.3975  \\
            NONDIA & 5000  & 25    & 64    & 0.0208  & 26    & 53    & 0.0176  & 26    & 65    & 0.0211  & 38    & 80    & 0.0264  & 20    & 55    & 0.0181  & 21    & 45    & 0.0153  \\
            PENALTY1 & 1000  & 50    & 84    & 0.0053  & 123   & 234   & 0.0140  & 96    & 153   & 0.0092  & 35    & 43    & 0.0028  & 132   & 251   & 0.0152  & 130   & 239   & 0.0144  \\
            POWER & 5000  & 2268  & 3520  & 0.6415  & 1761  & 3278  & 0.5582  & 2587  & 3974  & 0.7455  & 1990  & 3717  & 0.6410  & 1468  & 2301  & 0.4365  & 1283  & 2367  & 0.4311  \\
            POWELLSG & 1000  &    -   &  -     &  -     &  -     &   -    &   -    &    -   &  -    &   -    &  -     &     -  &  -     &  -     &   -    &  -     & 603   & 1147  & 0.0521  \\
            QUARTC & 1000  & 46    & 73    & 0.0033  & 40    & 44    & 0.0022  & 48    & 83    & 0.0036  & 36    & 46    & 0.0021  & 50    & 83    & 0.0038  & 31    & 35    & 0.0018  \\
            SCHMVETT & 5000  & 20    & 23    & 0.0313  & 37    & 38    & 0.0507  & 54    & 94    & 0.1237  & 35    & 42    & 0.0562  & 30    & 58    & 0.0766  & 39    & 40    & 0.0545  \\
            SENSORS & 1000  & 30    & 39    & 15.1534  & 22    & 26    & 8.1429  & 23    & 31    & 9.6644  & 20    & 24    & 6.1264  & 31    & 55    & 16.9737  & 24    & 28    & 7.2498  \\
            SINQUAD & 10000 & 32    & 44    & 0.0637  & 23    & 40    & 0.0566  & 19    & 33    & 0.0469  & 18    & 24    & 0.0344  & 29    & 46    & 0.0677  & 25    & 41    & 0.0604  \\
            SPARSQUR & 5000  & 29    & 44    & 0.0261  & 29    & 34    & 0.0200  & 26    & 46    & 0.0265  & 26    & 34    & 0.0196  & 39    & 70    & 0.0407  & 23    & 28    & 0.0172  \\
            SROSENBR & 5000  & 21    & 39    & 0.0100  & 18    & 28    & 0.0071  & 16    & 33    & 0.0084  & 21    & 35    & 0.0091  & 28    & 60    & 0.0152  & 33    & 53    & 0.0140  \\
            TOINTGOR & 50    & 191   & 308   & 0.0032  & 171   & 327   & 0.0033  & 201   & 324   & 0.0032  & 160   & 306   & 0.0031  & 134   & 227   & 0.0024  & 132   & 253   & 0.0026  \\
            TOINTPSP & 50    & 212   & 344   & 0.0016  & 252   & 481   & 0.0020  & 238   & 381   & 0.0015  & 231   & 441   & 0.0017  & 142   & 232   & 0.0010  & 181   & 343   & 0.0014  \\
            TOINTQOR & 50    & 44    & 59    & 0.0002  & 42    & 46    & 0.0002  & 44    & 59    & 0.0002  & 42    & 46    & 0.0002  & 37    & 42    & 0.0002  & 36    & 37    & 0.0001  \\
            VARDIM & 5000  & 369   & 639   & 0.1533  & 345   & 636   & 0.1499  & 369   & 639   & 0.1518  & 346   & 640   & 0.1500  & 369   & 639   & 0.1608  & 339   & 636   & 0.1575  \\
            VAREIGVL & 50    & 28    & 32    & 0.0003  & 27    & 28    & 0.0002  & 118   & 204   & 0.0016  & 104   & 202   & 0.0015  & 28    & 32    & 0.0003  & 27    & 28    & 0.0002  \\
            WOODS & 10000 & 84    & 156   & 0.0972  & 799   & 1534  & 0.9510  & 1590  & 2475  & 1.5742  & 1147  & 2171  & 1.3462  & 481   & 773   & 0.5111  & 72    & 107   & 0.0720  \\
            
            \bottomrule
        \end{tabular}%
    \end{lrbox}
    \scalebox{0.85}{\usebox{\tableboxaad}}
    \label{tab3}
\end{sidewaystable}%

\section*{Acknowledgement(s)}
This work was supported by the National Science Foundation of China under Grant No. 11571004.


\end{document}